\pdfoutput=1
\documentclass[english,11pt,oneside]{article}
\usepackage[T1]{fontenc}
\usepackage[left=2.5cm, right=2.5cm, top=2.5cm, bottom=2.5cm]{geometry}

\usepackage{mymacros}

\usepackage[authoryear]{natbib}

\numberwithin{figure}{section}

\theoremstyle{plain}
      \newtheorem{thm}{Theorem}[section] 
      \newtheorem{cor}{Corollary}[thm] 
      \newtheorem{lem}[thm]{Lemma}
      \newtheorem{prop}[thm]{Proposition} 

\theoremstyle{definition}
      \newtheorem{defn}[thm]{Definition}

\theoremstyle{remark}
      \newtheorem{re}{Remark}[thm]
      \newtheorem{com}{Comment}[thm] 
 


\begin{document}

\title{An Iterative Approximation of the Sublinear Expectation of an Arbitrary Function of $G$-normal
Distribution and the Solution to the Corresponding $G$-heat Equation}

\author{Yifan Li, Reg Kulperger}
\maketitle
\begin{abstract}

It has been a well-known problem in the $G$-framework that it
is hard to compute the sublinear expectation of the $G$-normal distribution
$\expt[\varphi(X)]$ when $\varphi$ is neither convex nor concave, if not
involving any PDE techniques to solve the corresponding $G$-heat
equation.  Recently, we have established an efficient iterative method able to compute the sublinear expectation of \emph{arbitrary} functions of the $G$-normal distribution, which directly applies the \emph{Nonlinear Central Limit Theorem} in the $G$-framework to a sequence of variance-uncertain random variables following the \emph{Semi-$G$-normal Distribution}, a newly defined concept with a nice \emph{Integral Representation}, behaving like a ladder in both theory and intuition, helping us climb from the ground of classical normal distribution to approach the peak of $G$-normal distribution through the \emph{iteratively maximizing} steps. 
The series of iteration functions actually produce the whole \emph{solution surface} of the $G$-heat equation on a given time grid. 

\end{abstract}

\newpage

\tableofcontents

\newpage

\section{Introduction}

(This is a short version and will be expanded to a longer version for general audience.) 


There is a long existing thinking gap between classical normal and $G$-normal distribution. 

For instance, the result
\[
\expt[\varphi(\GN(0,\varI))]\geq \sup _{v\in\sdInt} \lExpt[\varphi\big(N(0,\sigma^2))]
\]
indicates that the uncertainty set of $G$-normal distribution is much larger that a class of linear normal distributions with $\sigma\in\sdInt$. Especially, the strict inequality
\[
\expt[\big(\GN(0,\varI)\big)^3]>0=E[\big(N(0,\sigma^2)\big)^3]>-\expt[-\big(\GN(0,\varI)\big)^3]
\]
tells us that the ``skewness'' of $G$-normal distribution becomes uncertain so its ``symmetry'' is uncertain, which is quite counter-intuitive for a ``normal'' distribution. 
Is it possible to understand the $G$-normal distribution from our familiar classical normal distribution? In other words, is it possible to use the linear expectation of linear normal distribution (or the linear heat equation) to approach the sublinear expectation of $G$-normal distribution (or the nonlinear $G$-heat equation)? 
This paper gives positive answers to both of these questions by providing a ladder from the ground of $N(0,\sigma^2)$ to approach the peak of $\GN(0,\varI)$: the \emph{Semi-$G$-normal} distribution $\semiGN(0,\varI)$, which is a classical normal distribution scaled by a sublinear ``constant'' (the maximal distribution). 

\section{Preliminaries}
The $G$-expectation or sublinear expectation framework (also called the $G$-framework), motivated by the problems with \emph{ambiguity or uncertainty}, is a generalization of the linear probability framework. 
Similar to the Choquet expectation, the sublinear expectation can be represented as the supreme of a class of linear expectations.
Intuitively, the linear expectation mainly considers the ``average'', while the sublinear expectation focuses on the ``bound'' to create a interval to cover the uncertainty which is hard to be described by a certain distribution. Please turn to \cite{peng2007g, peng2008multi} for more details. 
\begin{defn}
A \emph{sublinear expectaion space} is defined as a triple $\exptSpace$.
 $\Omega$ is a given set
(also known as a sample space). 
$\mathcal{H}$ is a linear space of
real valued functions defined on $\Omega$ satisfying $c\in\mathcal{H}$
for each constant $c$ and $|X|\in\mathcal{H}$ if $X\in\mathcal{H}$,
which can be regarded as the space of random variables. 
$\expt$
is a \emph{sublinear expectation} which is a functional $\expt: \mathcal{H} \to \R^d$
satisfying: 
\begin{enumerate}
\item Monotonicity: $\expt[X]\geq\expt[Y]$ if $X\geq Y$;
\item Constant preserving: $\expt[c]=c$ for $c \in \R^d$;
\item Sub-additivity: For each $X,Y\in\mathcal{H}$, $\expt[X+Y]\leq\expt[X]+\expt[Y]$;
\item Positive homogeneity: $\expt[\lambda X]=\lambda\expt[X]$ for $\lambda\geq0$.
\end{enumerate}
If only monotonicity and constant preserving are satisfied, $\expt$ is called a nonlinear
expectation and $\exptSpace$ is called a nonlinear
expectation space. 

\end{defn} 
In the following context, we often use capital letters like $X\coloneqq (X_1,X_2,\dotsc,X_d),d\in \N_+$ to denote the random variables (or vectors) in $\mathcal{H}$. 
Meanwhile, if $X \in\mathcal{H}$, we also have $\varphi(X)\in\mathcal{H}$ for every $\varphi$ in $C_{l.Lip}(\R^d)$ which is the linear space of functions satisfying the locally Lipchistz property:
	\[
	|\varphi(x)-\varphi(y)|\leq C_\varphi |1+|x|^k+|y|^k||x-y|
	\]
	for $x,y\in\R^d$, some $k\in\N$ and $C_\varphi>0$ depending on $\varphi$. If not specified, we will always stay in the \emph{sublinear expectation space} $\exptSpace$ and the function space $C_{l.Lip}(\R^d)$ (or $C_{l.Lip}$ in short, which can be replaced by other spaces).
Our computation in this space is usually different from the linear expectation $\lExpt$ mainly because of the sub-additivity and positive homogeneity of $\expt$. Here are some useful tools to understand and deal with $\expt$. 	
\begin{prop} 
\label{mean-certain-1}
If $X$ has the mean-uncertainty: $\meanl<\meanr$ with $\meanl \coloneqq -\expt[-X]$ and $\meanr\coloneqq\expt[X]$, for $\lambda \in \R$, we will have 
\[
\expt[\lambda X] = \begin{cases}
	\lambda \expt[X] &\lambda\geq 0\\
	-\lambda \expt[-X] &\lambda<0
\end{cases} = \lambda^+\meanr - \lambda^-\meanl
\]
where $\lambda^+\coloneqq\max\{\lambda,0\}$ and $\lambda^-\coloneqq\max\{-\lambda,0\}$. 
\end{prop}
\begin{prop}
\label{mean-certain-2}
	If $X$ has the mean-certainty: $\meanl=\meanr\eqqcolon\mu$ namely $-\expt[-X]=\expt[X]=\mu$, for $\lambda \in \R$, we will have 
\[
\expt[\lambda X]=\lambda\expt[X](=\lambda\mu)
\]
and furthermore, 
\[
\expt[Y+\lambda X]=\expt[Y] + \lambda\expt[X]. 
\]
\end{prop}
\begin{proof}
	The first one is directly from the proposition \ref{mean-certain-1}. The second one comes from
\[
\expt[Y+\lambda X] \leq \expt[Y]+\expt[\lambda X]=\expt[Y]+\lambda\expt[X]
\]
and 
\[
\expt[Y+\lambda X]=\expt[Y-\lambda(-X)]\geq \expt[Y]-\expt[\lambda(-X)]=\expt[Y]+\lambda\expt[X].
\]
\end{proof}

\begin{prop}
	\label{holder}
	For $p,q>0$, $\frac 1p + \frac 1q =1$, we have
	\[
	\expt[|XY|]\leq (\expt[|X|^p])^{1/p}(\expt[|Y|^q])^{1/q}.
	\]
\end{prop}
\begin{prop}
	For $p\geq1$, we have
	\[
	(\expt[|X+Y|^p])^{1/p} \leq (\expt[|X|^p])^{1/p}+(\expt[|Y|^p])^{1/p}.
	\]
\end{prop}
	
We will use $\expt$ to redefine distributions and independence. 



\begin{defn}[Distributions]
	$\mathbb{F}_X$ is called the \emph{distribution} of $X$, which is a functional: 
	$\mathbb{F}_X[\varphi]\coloneqq \expt[\varphi(X)]:\varphi\in C_{l.Lip}(\R^d) \to \R$.
	$X$ and $Y$ are \emph{identically distributed}, denoted by $X\eqdistn Y$, iff for any $\varphi\in C_{l.Lip}$,
$\expt[\varphi(X)]=\expt[\varphi(Y)]$, namely, $\mathbb{F}_X[\varphi]=\mathbb{F}_Y[\varphi]$. 
A sequence $\{X_n\}_n^\infty$ \emph{converges in distribution} to $X$, denoted as $X_n \converge{\text{d}} X$, iff for any $\varphi\in C_{l.Lip}$,
$\lim_{n\to\infty}\expt[\varphi(X_n)]=\expt[\varphi(X)]$.
\end{defn}

\begin{defn}[Independence]
	$Y$ \emph{is independent from} $X$, denoted by $X\seqind Y$, iff for any $\varphi\in C_{l.Lip}$, 
	\[
	\expt[\varphi(X,Y)]=\expt[\expt[\varphi(x,Y)]_{x=X}].
	\]
	Intuitively, any realization of $X$ will have no effects on the distribution of $Y$. 
\end{defn}
\begin{defn}[i.i.d.]
	$\{X_i\}_{i=1}^\infty$ is \emph{i.i.d.} iff $X_{i+1}\eqdistn X_i$ and $(X_1,X_2,\dotsc,X_i)\seqind X_{i+1}$ for each $i\in\N_+$. 
\end{defn}
Let $\bar X$ be an independent copy of $X$, which means $\bar X \eqdistn X$ and $X\seqind \bar X$. 
\begin{defn}[Maximal Distribution]
	$X$ follows \emph{Maximal Distribution} iff, for any independent copy $\bar X$, 
	\[
	X+\bar{X}\eqdistn 2X.
	\]
	This is the sublinear version of a constant. A more specific definition is that $X$ follows the maximal distribution $M(\Gamma)$ if there exists a bounded, closed and convex set $\Gamma\subset\R^d$ such that for any $\varphi\in C_{l.Lip}(\R^d)$, 
	\[
	\mathbb{F}_X[\varphi]
	=\expt[\varphi(X)]=\max_{v\in\Gamma} \varphi(v).
	\]
	For $d=1$, we have $X\sim\Maximal\meanInt$ with \emph{mean-uncertainty}: $\meanl\coloneqq-\expt[-X]$ and $\meanr\coloneqq\expt[X]$. 
\end{defn}

\begin{defn}[$G$-normal Distribution]
	$X$ follows \emph{$G$-normal Distribution} iff, for any independent copy $\bar X$,
	\[
	X + \bar X \eqdistn \sqrt{2} X.
	\]	
	For $d=1$, we have $X\sim\GN(0,\varI)$ with \emph{variance-uncertainty}: $\sdl^2\coloneqq-\expt[-X^2]$ and $\sdr^2\coloneqq\expt[X^2]$.
\end{defn} 

Let $\myfield{S}_d$ denote the set of all real-valued $d\times d$ symmetric matrices. 

\begin{thm}[$G$-normal distribution Characterized by the $G$-heat Equation] $X$ follows the $d$-dimensional $G$-normal distribution, iff
	$u(t,x)\coloneqq \expt[\varphi(x+\sqrt{1-t}X)]$
	is the solution to \emph{the $G$-heat Equation} defined on $[0,\infty]\times\R^d$:
	\[
	u_t+G(D_x^2 u)=0,\,u|_{t=1}=\varphi
	\]
	where $G(\mymat{A})\coloneqq\frac 12\expt[\langle \mymat{A}X,X \rangle]:\myfield{S}_d \to \R$, which is a sublinear function characterizing the distribution of $X$. For $d=1$, we have $G(a)=\frac 12 (\sdr^2 a^+ -\sdl^2 a^-)$ and when $\sdl^2>0$, this is also called \emph{the Black-Scholes-Barenblatt equation} with volatility uncertainty. 
	
\end{thm}

\begin{re}
	We can use the function $G(\mymat{A})\coloneqq\frac 12\expt[\langle \mymat{A}X,X \rangle]$ to characterize the definition of $G$-normal distribution. In fact, $G(\mymat{A})$ can be further expressed as 
	\[G(\mymat{A})=\frac 12 \sup_{\dvar\in\dvarset} \trace[\mymat{A}\dvar]\] where $\dvarset=\{\mymat{B}\mymat{B}^T:\mymat{B}\in \myfield{S}_d\}$ is a collection of non-negative definite symmetric matrices which can be treated as the uncertainty set of the covariance matrices. 
\end{re}

\begin{defn}[$G$-normal distribution with Characterization] Let $X$ be any $d$-dimensional $G$-normal distributed random vector. To be specific, we say $X\sim\GN(\myvec 0, \dvarset)$ with the set of covariance matrices $\dvarset$ if its distribution is characterized by the $G$-heat equation with 
\[
G(\mymat{A})\coloneqq\frac 12\expt[\langle \mymat{A}X,X \rangle]=\frac 12 \sup_{\dvar\in\dvarset} \trace[\mymat{A}\dvar].\]
In other words, $\myset V$ is the set corresponding to the $G$ function characterizing the distribution of $X$.  In order to show the \emph{covariance-uncertainty} of $\GN(\myvec 0,\dvarset)$, we can expand the details of $\dvarset$ as 
\[
\dvarset\coloneqq\{\mymat{V}=\left(\rho_{ij}\sigma_{i}\sigma_{j}\right)_{d\times d}:\sigma_{i}^{2}\in[\sdl_{i}^{2},\sdr_{i}^{2}],\rho_{ij}=\rho_{ji}=\begin{cases}
1 & i=j\\
\in[\underline{\rho}_{ij},\overline{\rho}_{ij}] & i\neq j
\end{cases},i,j=1,2,\dotsc,d \}.
\]

When $d=1$, we say $X\sim \GN(0,\varI)$ with the variance interval $\varI$ if its distribution is characterized by the $G$-heat equation with 
\[
G(a) \coloneqq \frac 12 \expt[aX^2] = \frac 12 (\sdr^2 a^+ -\sdl^2 a^-).
\]

\end{defn}

	

\begin{thm}[The Nonlinear Central Limit Theorem by Peng]
\label{$G$-CLT}
	Consider a sequence of i.i.d. $\{X_i\}_{i=1}^\infty$ with mean-certainty $\expt[X_1]=-\expt[-X_1]=\myvec 0$. Let $X$ be a $G$-normal distributed random variable characterized by the function $G(\mymat{A})\coloneqq \frac 12 \expt[\langle \mymat{A}X_1,X_1 \rangle]$. Then for any $\varphi \in C_{l.Lip}$, 
	\[
	\lim_{n\to\infty} \expt[\varphi(\frac{1}{\sqrt n}\sum_{i=1}^{n} X_i)] = \expt[\varphi(X)].
	\]
	For $d=1$, let $\sdl^2\coloneqq-\expt[-X_1^2]$ and $\sdr^2\coloneqq\expt[X_1^2]$. Then we have $\frac{1}{\sqrt n}\sum_{i=1}^{n} X_i$ converges in distribution to $X\sim \GN(0,\varI)$ . 
\end{thm}

We will call Theorem \ref{$G$-CLT} the nonlinear CLT in short. We also have the convergence rate of nonlinear CLT from \cite{song2017normal}. 

\begin{thm}[The Convergence Rate of Nonlinear CLT] 
\label{$G$-CLT-Rate}
Under the setting of Theorem \ref{$G$-CLT} when $d=1$, for bounded and Lipschitz continuous $\varphi$ (i.e. for any $x,y \in \R$, $|\varphi(x)|\leq M$ and $|\varphi(x)-\varphi(y)|\leq C_\varphi|x-y|$), there exist $\alpha\in(0,1)$ depending on $\sdl$ and $\sdr$, and $C_{\alpha,G}>0$ depending on $\alpha$, $\sdl$ and $\sdr$ such that 
	\[
	\sup_{C_\varphi \leq 1}\left|\expt[\varphi(\frac{1}{\sqrt n}\sum_{i=1}^{n} X_i)]-\expt[\varphi(X)]\right|\leq C_{\alpha,G}\frac{\expt[|X_1|^{2+\alpha}]}{n^{\frac{\alpha}{2}}} = \rate.
	\]
\end{thm}


\section{Main Theoretical Results\label{sec:Main-results}}

The iterative algorithm is based on a new concept, the \emph{Semi-$G$-normal distribution} $\semiGN(0,\varI)$ with a smaller uncertainty set than the $G$-normal distribution. This new distribution has a nice integral representation and a general connection with the $G$-normal distribution, that is the CLT (central limit theorem) in the $G$-framework. For reader's convenience, we will first give the results in one dimension and they can be easily extended to multi-dimensional cases. 
\subsection{The Semi-$G$-normal Distribution in one dimension}


\begin{defn}[The Semi-$G$-normal Distribution]
$W$ follows the \emph{Semi-$G$-normal distribution}, denoted
as, $W\sim\hat{N}(0,\varI)$, if there exist $Y\sim N(0,[1,1])$ and
$Z\sim M[\sdl,\sdr]$, $\sdr\geq\sdl\geq0$ with independent relation $Z\seqind Y$, such that 
\[
W=Z\cdot Y
\]
where ``$\cdot$'' is the number multiplication (which can be omitted) and the direction of independence here cannot be reversed. 
\end{defn}

\begin{re}
	$Y\sim\GN(0,[1,1])$
can be regarded as the classical standard normal distribution $N(0,1)$ since the corresponding $G$-heat equation will be reduced to classical heat equation when $\underline\sigma$ and $\overline\sigma$ coincide. 
\end{re}
\begin{re}[the mean and variance of $W$] It is not hard to show
that it has a certain zero mean: 
\[
\expt[W] =\expt[ZY]
 =\expt[\expt[zY]_{z=Z}]
 =\expt[E[zY]_{z=Z}] = 0
\]
and $-\expt[-W]=0$. For the variance, we have 
\[
\expt[W^{2}] =\expt[Z^{2}Y^{2}]
 =\expt[\expt[z^{2}Y^{2}]_{z=Z}]
 =\expt[E[z^{2}Y^{2}]_{z=Z}]
 =\expt[(z^{2}\cdot1)_{z=Z}]=\max_{z\in\sdInt}z^{2}=\sdr^{2}
\]
and similarly, 
\[
-\expt[-W^{2}] =\expt[-Z^{2}Y^{2}]
 =-\expt[E[-z^{2}Y^{2}]_{z=Z}]
 =-\max_{z\in\sdInt}(-z^{2})=\min_{z\in\sdInt}z^{2}=\sdl^{2}.
\]

\end{re}

\begin{thm}[The Integral Representation of the Semi-$G$-normal Distribution] \label{integral-rep} Let $W\sim\semiGN(0,\varI)$ then for any $\varphi\in C_{l.Lip}(\R)$,
we have 
\[
\expt[\varphi(W)] = \max_{z\in[\sdl,\sdr]}E[\varphi(N(0,z^{2}))]=\max_{z\in[\sdl,\sdr]}\int_{-\infty}^{\infty}\frac{1}{\sqrt{2\pi}}e^{-y^{2}/2}\varphi(zy)\diff y.
\]
\end{thm}

\begin{proof} This is quite straightforward because: 
\[
\expt[\varphi(W)] =\expt[\varphi(YZ)]
 =\expt[\expt[\varphi(Yz)]_{z=Z}]\eqqcolon\expt[G(Z)]
\]
where $G(z)\coloneqq\expt[\varphi(Yz)]=E[\varphi(N(0,z^{2}))]$ can
be proved to be in $C_{l.Lip}$ based on $\varphi\in C_{l.Lip}$
. Specifically, we have 
\begin{align*}
|G(x)-G(y)| & =|E[\varphi(x\cdot Z)-\varphi(y\cdot Z)]|\\
 & \leq E[C_\varphi(1+|xZ|^k+|yZ|^k)|Z|\cdot|x-y|]\\
 & =C_\varphi\cdot E[|Z|+|Z|^{k+1}|x|^k+|Z|^{k+1}|y|^k]|x-y|\\
 & \leq C_G(1+|x|^k+|y|^k)|x-y|
\end{align*}
where $C_G=C_\varphi\max\{E[|Z|],E[|Z|^{k+1}]\}$.
Therefore, 
\begin{align*}
\expt[\varphi(W)] & =\expt[G(Z)]=\max_{z\in\sdInt}G(z)\\
 & =\max_{z\in\sdInt}E[\varphi(N(0,z^{2}))]\\
 & =\max_{z\in\sdInt}\int_{-\infty}^{\infty}\frac{1}{\sqrt{2\pi}}e^{-y^{2}/2}\varphi(zy)\,dy.\qedhere
\end{align*}
\end{proof}

\begin{re}[Why is it called a ``semi'' one?]\label{why-semi}
	The comparison theorem of parabolic PDEs ($G$-heat vs Heat Equations) tells us that
\[
\expt[\varphi(\GN(0,\varI))]\geq\max_{v\in\sdInt}E[\varphi(N(0,v^{2}))]=\expt[\varphi(\semiGN(0,\varI))]
\]
whose inequality is mostly strict (like $\varphi(x)=x^3$) and becomes equal when $\varphi$ is convex or concave. From the representation theorem of $\expt$: 
$\expt[\varphi(X)]=\sup_{P\in\mathcal Q} \lExpt_P[\varphi(X)]$, we have the intuition that $\mathcal{Q}_\text{Semi-$G$-normal} \subset \mathcal{Q}_\text{$G$-normal}$ where $\mathcal{Q}_\text{Semi-$G$-normal}$ consists of measures corresponding to classical normal distributions with $\sigma\in\sdInt$.
\end{re}

\begin{re} Intuitively, the $G$-normal distribution is more ``uncertain'' than the semi-$G$-normal distribution. 
	Explicitly speaking, we already have the following representation for the $G$-normal distributed $X\sim\GN(0,\varI)$ (see \cite{denis2011function}): 
\[
\expt[\varphi(X)]=\sup_{\theta\in A^{\Theta}}E_{P}[\varphi(\int_{0}^{1}\theta_{s}dB_{s})]
\]
where $\Theta=\sdInt$, $A^{\Theta}\coloneqq\{\theta:\theta_{t}\in\Theta,\text{for }t\in[0,1]\}$,
the set of all processes valuing in $\sdInt$ in the time range $[0,1]$ and $B$ is the classical Brownian motion in $\lExptSpace$. Meanwhile, from the integral representation (Theorem \ref{integral-rep}),  
we can show that for $W\sim\semiGN(0,\varI)$, 
\[
\expt[\varphi(W)]=\sup_{\theta\in A^{\Theta}}E_{P}[\varphi(\int_{0}^{1}\bar{\theta}dB_{s})]
\]
where $\bar{\theta}=\int_{0}^{1}\theta_{s}ds (\in \sdInt)$, the average of the
process $\theta$ over the time interval $[0,1]$. To summarize, the $G$-normal distribution has the uncertainty set consisting of \emph{all processes} valuing in $\Theta$, while the semi-$G$-normal distribution only has the set made up of \emph{all constant processes} valuing in $\Theta$. 
\end{re}
\begin{cor}[the connection with the $G$-normal distribution] \label{conn-G}
When $\varphi$ is convex or concave and $\varphi\in C^2(\R)$, for $X\sim N(0,\varI)$ and
$W\sim\hat{N}(0,\varI)$, we have 
\[
\expt[\varphi(X)]=\expt[\varphi(W)].
\]
\end{cor}
\begin{proof} The well-known
integral representation of $G$-normal distribution under convexity (or
concavity) directly comes from the solution of the classical heat
equation because $u(t,x)$ will be convex (or concave, respectively)
to $x$ then $u_{xx}\geq0$ (or $\leq0$, respectively), giving us
\[
\expt[\varphi(X)]=\begin{cases}
E[\varphi(N(0,\sdr^{2}))] & \varphi\text{ is convex}\\
E[\varphi(N(0,\sdl^{2}))] & \varphi\text{ is concave}
\end{cases}.
\]
For the semi-$G$-normal distribution, by using its representation with $G(z)\coloneqq \lExpt[\varphi(zY)](z\in\sdInt)$ and $Y\sim N(0,1)$, we only need to prove that 
\[
\expt[\varphi(W)]=\max_{z\in\sdInt}G(z)= \begin{cases}
G(\sdr) & \varphi\text{ is convex}\\
G(\sdl) & \varphi\text{ is concave}
\end{cases}.
\]
First of all, $\varphi$ has the Taylor expansion 
\[
\varphi(x)=\varphi(0)+\varphi^{(1)}(0)x+\varphi^{(2)}(\xi_{x})\frac{x^{2}}{2}
\]
where $\xi_x \in (0,x)$. 

1. When $\varphi$ is convex, $\varphi^{(2)}(\xi_x)\geq 0$.
The Taylor expansion tells us that:
\begin{align*}
G(z) & = E[\varphi(zY)]\\
 & =E[\varphi(0)+\varphi^{(1)}(0)zY+\varphi^{(2)}(\xi_{zY})\frac{z^{2}}{2}Y^{2}]\\
 & =\varphi(0)+\frac{1}{2}E[\varphi^{(2)}(\xi_{zY})(zY)^{2}]
\end{align*}
where $\xi_{zY}\in(0,zY)$ is a random variable depending on $Y$. Let $M\coloneqq zY\sim N(0,z^{2})$,
then 
\[
K(z)\coloneqq \lExpt[\varphi^{(2)}(\xi_{zY})(zY)^{2}]=\lExpt[\varphi^{(2)}(\xi_{M})M^{2}]=\int\frac{1}{\sqrt{2\pi}}\exp(-\frac{m^{2}}{2z^{2}})\varphi^{(2)}(\xi_{m})m^{2}\diff m.
\]
In order to consider
the monotone property of $K(z)$, work on its derivative: 
\begin{align*}
K^{\prime}(z) & =\frac{d}{dz}\int\frac{1}{\sqrt{2\pi}}\exp(-\frac{m^{2}}{2z^{2}})\varphi^{(2)}(\xi_{m})m^{2}\diff m\\
 & =\int\frac{1}{\sqrt{2\pi}}\left[\frac{d}{dz}\exp(-\frac{m^{2}}{2z^{2}})\right]\varphi^{(2)}(\xi_{m})m^{2}\diff m\\
 & =\int\frac{1}{\sqrt{2\pi}}\underbrace{\left[\frac{m^{2}}{z^{3}}\exp\left(-\frac{m^{2}}{2z^{2}}\right)\right]}_{\geq0\text{ for }z\in\sdInt}\underbrace{\vphantom{\left[\frac{1}{2z}\exp\left(-\frac{m^{2}}{2z^{2}}\right)\right]}\varphi^{(2)}(\xi_{m})\,m^{2}}_{\geq0}\diff m \geq 0.
\end{align*}
This tells us $K(z)$ is increasing with respect to $z\in \sdInt$, then
$K(z)$ reaches its maximum at $z=\sdr$. Hence,
\[
\expt[\varphi(W)] = \max_{z\in\sdInt}G(z) = \max_{z\in\sdInt}(\varphi(0)+\frac{K(z)}{2})=G(\sdr).
\]

2. When $\varphi$ is concave, then $-\varphi$ is convex. replace the
$\varphi$ above with $-\varphi$ and repeat the same procedure, we
have 
\begin{align*}
-G(z) & =-E[\varphi(zY)]\\
 & =E[(-\varphi)(zY)]\\
 & =(-\varphi)(0)+\frac{z^{2}}{2}\underbrace{E[(-\varphi)^{(2)}(\xi_{zY})Y^{2}]}_{K(z)\geq0}
\end{align*}
and 
\[
K^{\prime}(z)=\int\frac{1}{\sqrt{2\pi}}\underbrace{\left[\frac{m^{2}}{z^{3}}\exp\left(-\frac{m^{2}}{2z^{2}}\right)\right]}_{\geq0\text{ for }z\in \sdInt}\underbrace{\vphantom{\left[\frac{m^{2}}{z^{3}}\exp\left(-\frac{m^{2}}{2z^{2}}\right)\right]}(-\varphi)^{(2)}(\xi_{m})\,m^{2}}_{\geq0}\diff m \geq 0
\]
Hence, $-G(z)$ is increasing with respect to $z$, that is, $G(z)$
is decreasing according to $z$. Therefore, 
\[
\expt[\varphi(X)]=\max_{z\in\sdInt}G(z)=G(\sdl).\qedhere
\]
\end{proof}
The initial motivation of the \emph{semi-$G$-normal} distribution is that we want to create a tool or ladder to help us better understand and handle the $G$-normal distribution, especially based on what we already know about the classical normal distribution, which turns out to be feasible from the nice properties of semi-$G$-normal distribution and thanks to the constructed theory in $G$-framework (like the nonlinear CLT by Peng). The following result is one of the exciting results from the semi-$G$-normal distribution to better understand and compute the expectation of the $G$-normal distribution. 

\subsection{The Iterative Approximation of the $G$-normal Distribution in one dimension}

\begin{lem}[General connection between the Semi-$G$-normal and the
$G$-normal distribution] \label{nl-CLT-semiG} 
Consider a sequence of i.i.d.\ $\{W_{i}\}_{i=1}\sim\hat{N}(0,\varI)$
and $X\sim N(0,\varI)$, then we have 
\[
\lim_{n\to\infty}\expt[\varphi(\frac{1}{\sqrt{n}}\sum_{i=1}^{n}W_{i})]=\expt[\varphi(X)],\;\forall\varphi\in C_{l.Lip}(\R).
\]
In other words, $\frac{1}{\sqrt{n}}\sum_{i=1}^{n}W_{i}$
converges in distribution to the $G$-normal distributed $X$. 

\end{lem}

Lemma \ref{nl-CLT-semiG} is a direct result of the CLT (Theorem \ref{$G$-CLT}) in the $G$-framework. 

\begin{thm}[The Iterative Approximation of the $G$-normal Distribution] \label{iterate-G} Consider a $G$-normal distributed random variable
$X\sim N(0,\varI)$. For any $\varphi\in C_{l.Lip}(\R)$ and integer $n \geq 1$, consider the series of iteration functions 
$\{\varphi_{i,n}\}_{i=1}^n$ with initial function $\varphi_{0,n}(x)\coloneqq\varphi(x)$ and iterative relation: 
\[
\varphi_{i+1,n}(x)\coloneqq \max_{v\in[\sdl,\sdr]}E[\varphi_{i,n}(N(x,v^2/n))],i=0,1,\dotsc,n-1.
\]
The final iteration function for a given $n$ is $\varphi_{n,n}$. As $n\to\infty$, we have $\varphi_{n,n}(0)\to\expt[\varphi(X)]$. 

\end{thm}

\begin{proof} Set the initial function $\varphi_{0,n}(x)\coloneqq\varphi(x)$
and the iteration 
\[
\varphi_{i,n}(x)\coloneqq\max_{v\in[\sdl,\sdr]}E[\varphi_{i-1,n}(N(x,\frac{v^{2}}{n}))].
\]
In order to use the integral representation of the Semi-$G$-normal distribution
(Theorem \ref{integral-rep}) in the next stage, we want each iteration
function to be in the function space $C_{l.Lip}$. For convenience,
we omit the subscript $n$ for a while and let $\varphi_{i}\coloneqq\varphi_{i,n}$.
We also know that the optimal $v$ will be some value in $\sdInt$
depending on $x$, i.e. 
\[
\varphi_{i}(x)=E[\varphi_{i-1}(N(x,\frac{v_{x}^{2}}{n})],v_{x}\in\sdInt.
\]
 By induction, we only need to show that given $\varphi_{i-1}\in C_{l.Lip}$,
we also have $\varphi_{i}\in C_{l.Lip}$, for $i=1,2,\dotsc,n$. Suppose
for $\varphi_{i-1}$ we have a constant $C_{i-1}$ and an positive
integer $k$ such that 
\[
|\varphi_{i-1}(x)-\varphi_{i-1}(y)|\leq C_{i-1}(1+|x|^{k}+|y|^{k})|x-y|.
\]
Let $Z\sim N(0,1)$ and consider
\begin{align*}
|\varphi_{i}(x)-\varphi_{i}(y)| & =|E[\varphi_{i-1}(N(x,\frac{v_{x}^{2}}{n})\text{]}-E[\varphi_{i-1}(N(y,\frac{v_{y}^{2}}{n})\text{]}|\\
 & =|E[\varphi_{i-1}(x+\frac{v_{x}}{\sqrt{n}}Z)\text{]}-E[\varphi_{i-1}(y+\frac{v_{y}}{\sqrt{n}}Z)\text{]}|\\
 & \leq E|\varphi_{i-1}(x+\frac{v_{x}}{\sqrt{n}}Z)-\varphi_{i-1}(y+\frac{v_{y}}{\sqrt{n}}Z)|\\
 & \leq E[C_{i-1}(1+|x+\frac{v_{x}}{\sqrt{n}}Z|^{k}+|y+\frac{v_{y}}{\sqrt{n}}Z|^{k})|(x+Z)-(y+Z)|]\\
 & =C_{i-1}|x-y|(1+E|x+\frac{v_{x}}{\sqrt{n}}Z|^{k}+E|y+\frac{v_{y}}{\sqrt{n}}Z|^{k}).
\end{align*}
For given $\omega\in\Omega$, let $z\coloneqq Z(\omega)$, we have 
\begin{align*}
|x+\frac{v_{x}}{\sqrt{n}}z|^{k} & =|(x+\frac{v_{x}}{\sqrt{n}}z)^{k}|\leq\sum_{j=1}^{k}\binom{k}{j}|x|^{j}|\frac{v_{x}}{\sqrt{n}}z|^{k-j}\\
 & \leq\sum_{j=1}^{k}\binom{k}{j}\max\{|x|^{k},|\frac{v_{x}}{\sqrt{n}}z|^{k}\}\leq2^{k}(|x|^{k}+|\frac{v_{x}}{\sqrt{n}}|^{k}|z|^{k})\leq2^{k}(|x|^{k}+|\frac{\sdr}{\sqrt{n}}|^{k}|z|^{k})
\end{align*}
where the $2^k$ can be improved to $\max\{2^{k-1},1\}$ but they are both constants so $2^k$ is good enough for us. Then $|x+\frac{v_{x}}{\sqrt{n}}Z(\omega)|^{k}\leq2^{k}(|x|^{k}+|\frac{\sdr}{\sqrt{n}}|^{k}|Z(\omega)|^{k})$;
taking expectation on both sides, we have 
\[
E|x+\frac{v_{x}}{\sqrt{n}}Z|^{k}\leq2^{k}(|x|^{k}+|\frac{\sdr}{\sqrt{n}}|^{k}E|Z|^{k}).
\]
Therefore, 
\begin{align*}
|\varphi_{i}(x)-\varphi_{i}(y)| & \leq C_{i-1}|x-y|(1+E|x+\frac{v_{x}}{\sqrt{n}}Z|^{k}+E|y+\frac{v_{y}}{\sqrt{n}}Z|^{k})\\
 & \leq C_{i-1}|x-y|(1+2^{k+1}|\frac{\sdr}{\sqrt{n}}|^{k}E|Z|^{k}+2^{k}|x|^{k}+2^{k}|y|^{k})\\
 & \leq C_{i}(1+|x|^{k}+|y|^{k})|x-y|
\end{align*}
where $C_{i}\coloneqq C_{i-1}\cdot\max\{1+2^{k+1}|\frac{\sdr}{\sqrt{n}}|^{k}E|Z|^{k},2^{k}\}$. 

Considering a sequence of nonlinear i.i.d.\ $\{W_{i}\}_{i=1}\sim\hat{N}(0,\varI)$
and $W_{i,n}\coloneqq\frac{1}{\sqrt{n}}W_{i}$, since $\varphi_{i}\in C_{l.Lip}$,
we can apply the integral representation of the Semi-$G$-normal distribution
at each step. Then we have 
\begin{align*}
\expt[\varphi(\frac{1}{\sqrt{n}}\sum_{i=1}^{n}W_{i})] & =\expt[\varphi_{0,n}(\sum_{i=1}^{n}W_{i,n})]\\
 & =\expt\big[\expt[\varphi_{0,n}(\sum_{i=1}^{n-1}w_{i,n}+W_{n,n})]_{w_{i,n}=W_{i,n}\,i=1,2,\dotsc,n-1}\big]\\
 & =\expt\Big[\big[\max_{v_{n}\in[\sdl,\sdr]}E[\varphi_{0,n}(\sum_{i=1}^{n-1}w_{i,n}+N(0,\frac{v_{n}^{2}}{n}))]\big]_{w_{i,n}=W_{i,n},\,i=1,2,\dotsc,n-1}\Big]\\
 & =\expt\Big[\big[\max_{v_{n}\in[\sdl,\sdr]}E[\varphi_{0,n}(N(\sum_{i=1}^{n-1}w_{i,n},\frac{v_{n}^{2}}{n}))]\big]_{w_{i,n}=W_{i,n},\,i=1,2,\dotsc,n-1}\Big]\\
 & =\expt[\varphi_{1,n}(\sum_{i=1}^{n-1}W_{i,n})]\\
 & =\expt\Big[\big[\max_{v_{n-1}\in[\sdl,\sdr]}E[\varphi_{1,n}(N(\sum_{i=1}^{n-2}w_{i,n},\frac{v_{n-1}^{2}}{n}))]\big]_{w_{i,n}=W_{i,n}\,i=1,2,\dotsc,n-2}\Big]\\
 & =\cdots\\
 & =\expt\Big[\big[\max_{v_{2}\in[\sdl,\sdr]}E[\varphi_{n-2,n}(N(w_{1,n},\frac{v_{2}^{2}}{n}))]\big]_{w_{1,n}=W_{1,n}}\Big]\\
 & =\expt[\varphi_{n-1,n}(W_{1,n})]\\
 & =\max_{v_{1}\in[\sdl,\sdr]}E[\varphi_{n-1,n}(N(0,\frac{v_{1}^{2}}{n}))] = \varphi_{n,n}(0)
\end{align*}
According to Lemma \ref{nl-CLT-semiG}, we have 
\[
\expt[\varphi(X)]=\lim_{n\rightarrow\infty}\expt[\varphi(\frac{1}{\sqrt{n}}\sum_{i=1}^{n}W_{i})]=\lim_{n\rightarrow\infty}\varphi_{n,n}(0).\qedhere
\]

\end{proof}

\begin{re}
	From the proof, we note that the iteration function can also be expressed as the sublinear expectation of the semi-$G$-normal distribution (letting $W_0\coloneqq 0$): 
	\[
	\varphi_{i,n}(x)=\expt[\varphi(x+\sum_{j=0}^{i}\frac{W_{n-j}}{\sqrt{n}})]=\expt[\varphi(x+\sum_{j=0}^{i}\frac{W_j}{\sqrt{n}})]
	\]
	for $i=0,1,\dotsc,n$. 
\end{re}

Furthermore, the following result shows that $\{\varphi_{i,n}\}_{i=0}^n$ produces the whole solution surface of the corresponding $G$-heat equation on a given time grid. The proof involves the  convergence rate of the nonlinear CLT (Theorem \ref{$G$-CLT-Rate}) from \cite{song2017normal}, which considers the functions $\varphi$ being bound and Lipschitz. If this is not the only function space, we can adapt the above result to the space $\fspace$ (locally Lipschitz) to get the Corollary \ref{converge-rate}.

\begin{cor}
\label{converge-rate}
 Consider the $G$-heat equation defined on $[0,\infty]\times\R$:
	\[
	u_t+G(u_{xx})=0,\,u|_{t=1}=\varphi
	\]
where $G(a)\coloneqq \frac{1}{2}\expt[aX^{2}]=\frac{1}{2}(\sdr^{2}a^{+}-\sdl^{2}a^{-})$ and $\varphi\in C_{l.Lip}(\R)$.
Then for the iterations $\{\varphi_{i,n}\}_{i=0}^n$ in Theorem \ref{iterate-G}, for each $p\in(0,1]$, we have
\[
|u(1-p,x) - \varphi_{\lfloor np\rfloor,n}(x)|=|\expt[\varphi(x+\sqrt p X)]-\expt[\varphi(x+\sum_{i=0}^{\lfloor np\rfloor}\frac{W_i}{\sqrt{n}})]| = C_\varphi(1+|x|^k) O(\frac{1}{(np)^{\alpha/2}}).
\] 
where $\floor{np}$ is the floor (or integer) part of $np$.
When $p=0$, we have $u(1,x)=\varphi(x)=\varphi_{0,n}(x)$. 
\end{cor}

\begin{re}
	If $\varphi$ is required to be bounded and Lipschitz, we may remove the local factor $(1+|x|^k)$ from the error part to get a uniform error approximation regardless of $x$. 
\end{re}

\begin{proof}
For each $p\in(0,1)$, consider the error when approximating $u(1-p,x)$, which can be approximated by 
\[
u(1-\frac{\lfloor np\rfloor}{n},x)=\expt[\varphi(x+\sum_{j=0}^{\lfloor np\rfloor}\frac{X_{j}}{\sqrt{n}})]\approx\expt[\varphi(x+\sum_{i=0}^{\lfloor np\rfloor}\frac{W_i}{\sqrt{n}})]=\varphi_{\lfloor np\rfloor,n}(x).
\]
 Specifically speaking, we intend to work on the error
\[
|u(1-p,x)-\varphi_{\floor{np},n}(x)|\leq\underbrace{|u(1-p,x)-u(1-\frac{\floor{np}}{n},x)|}_{(1)}+\underbrace{|u(1-\frac{\floor{np}}{n},x)-\varphi_{\floor{np},n}(x)|}_{(2)}.
\]
Before diving into these two parts, we can prepare the converging
property of $p_{n}\coloneqq\frac{\lfloor np\rfloor}{n}$. Actually, the inequality 
\[
p_{n}=\frac{\floor{np}}{n}\leq\frac{np}{n}\leq\frac{\floor{np}+1}{n}=p_{n}+\frac{1}{n},
\]
tells us that 
\[
|p_{n}-p|<\frac{1}{n}.
\]
The $(1)$ part involves the continuity of $u$ on the time dimension
specified by doing the Taylor expansion:
\begin{align*}
(1) & \leq|u_{t}(1-p,x)||p_{n}-p|+\underbrace{O(|p_{n}-p|^{2})}_{O(\frac{1}{n^{2}})}.
\end{align*}
We are looking for bound of $|u_{t}|$. Fortunately, according to
the properties of the solution to the $G$-heat equation (see \cite{pao2012nonlinear} and \cite{wang1992regularity}), there
exist constant $C>0$ and $\beta>0$, such that 
\[
|u_{t}(t,x)-u_{t}(s,x)|\leq C|t-s|^{\frac{\beta}{2}}.
\]
By letting $s=0$ and $c_{x}\coloneqq|u_{t}(0,x)|$, we have 
\[
|u_{t}(t,x)|\leq|u_{t}(t,x)-u_{t}(0,x)|+|u_{t}(0,x)|\leq C|t|^{\frac{\beta}{2}}+c_{x}.
\]
Therefore, for a fixed $x$, we have
\[
(1)\leq(C|1-p|^{\frac{\beta}{2}}+c_{x})|p_{n}-p|+O(\frac{1}{n^{2}})=c_xO(\frac{1}{n}).
\]
The $(2)$ part can be rewritten as follows:
\begin{align*}
(2) & =|\expt[\varphi(x+\frac{1}{\sqrt{n}}\sum_{j=0}^{\floor{np}}X_{j})]-\expt[\varphi(x+\frac{1}{\sqrt{n}}\sum_{j=0}^{\floor{np}}W_{j})]|\\
 & =|\expt[\varphi(x+\sqrt{\frac{\floor{np}}{n}}\frac{1}{\sqrt{\floor{np}}}\sum_{j=0}^{\lfloor np\rfloor}X_{j})]-\expt[\varphi(x+\sqrt{\frac{\floor{np}}{n}}\frac{1}{\sqrt{\floor{np}}}\sum_{j=0}^{\lfloor np\rfloor}W_{j})]|.
\end{align*}
By letting $\mathbf{X}_{np}\coloneqq\frac{1}{\sqrt{\floor{np}}}\sum_{j=0}^{\lfloor np\rfloor}X_{j}$
and $\mathbf{W}_{np}\coloneqq\frac{1}{\sqrt{\floor{np}}}\sum_{j=0}^{\lfloor np\rfloor}W_{j}$
, we have
\begin{align*}
(2) & =|\expt[\varphi(x+\sqrt{p_{n}}\mathbf{X}_{np})]-\expt[\varphi(x+\sqrt{p_{n}}\mathbf{W}_{np})]|\\
 & \leq|\expt[\varphi(x+\sqrt{p_{n}}\mathbf{X}_{np})]-\expt[\varphi(x+\sqrt{p}\mathbf{X}_{np})]|\\
 & +|\expt[\varphi(x+\sqrt{p_{n}}\mathbf{W}_{np})]-\expt[\varphi(x+\sqrt{p}\mathbf{W}_{np})]|\\
 & +|\expt[\varphi(x+\sqrt{p}\mathbf{X}_{np})]-\expt[\varphi(x+\sqrt{p}\mathbf{W}_{np})]|\\
 & \coloneqq(2)_{1}+(2)_{2}+(2)_{3}
\end{align*}
where $(2)_{1}+(2)_{2}$ involves the continuity of $\varphi$ and
the shrinking speed of $|p_{n}-p|$ and $(2)_{3}$ is exactly fitted
into the convergence rate of nonlinear CLT. In $(2)_{1}$ or $(2)_{2}$,
we do not need to worry about the random variables since $\mathbf{X}_{np}\eqdistn X\eqdistn\GN(0,\varI)$
and $\mathbf{W}_{np}\converge\dist X$ as $n\to\infty$ from nonlinear
CLT. Hence, let us work on a general expression by replacing \textbf{$\mathbf{X}_{np}$
}and $\mathbf{W}_{np}$ by $Z_{n}$ satisfying $Z_{n}\converge\dist X$.
We know $\varphi\in C_{l.Lip}$ satisfying, for $x,y\in\R$,
\[
|\varphi(x)-\varphi(y)|\leq C_{\varphi}(1+|x|^{k}+|y|^{k})|x-y|
\]
with $C_{\varphi}>0$ and $k\in\mathbb{N}$. When $k=0$, $\varphi$
is uniformly Lipschitz, then 
\begin{align*}
(2)_{1}\text{or}(2)_{2} & =|\expt[\varphi(x+\sqrt{p_{n}}Z_{n})]-\expt[\varphi(x+\sqrt{p}Z_{n})]|\\
 & \leq\expt[|\varphi(x+\sqrt{p_{n}}Z_{n})-\varphi(x+\sqrt{p}Z_{n})|]\\
 & \leq\expt[C_{\varphi}|\sqrt{p_{n}}-\sqrt{p}||Z_{n}|]=C_{\varphi}|\sqrt{p_{n}}-\sqrt{p}|\expt[|Z_{n}|]
\end{align*}
For the expectation part in the last line, since $Z_{n}\converge\dist X$,
we can make $n$ large enough so that
\[
\expt[|Z_{n}|]\leq\expt[|X|]+1\eqqcolon K_{0}.
\]
For the $p_{n}$ part, the Taylor expansion of $\sqrt{x}$ at $x=p$
tells us that 
\[
|\sqrt{p_{n}}-\sqrt{p}|\leq\frac{1}{2\sqrt{p}}|p_{n}-p|+O(|p_{n}-p|^{2})=O(\frac{1}{n\sqrt{p}}).
\]
Hence, when $k=0$,
\[
(2)_{1}\text{or}(2)_{2}\leq C_{1}O(\frac{1}{n\sqrt{p}})
\]
with $C_{1}=C_{\varphi}K_{0}$. When $k\geq1$, we have 
\begin{align*}
(2)_{1}\text{or}(2)_{2} & \leq\expt[|\varphi(x+\sqrt{p_{n}}Z_{n})-\varphi(x+\sqrt{p}Z_{n})|]\\
 & \leq\expt[C_{\varphi}(1+|x+\sqrt{p_{n}}Z_{n}|^{k}+|x+\sqrt{p}Z_{n}|^{k})|\sqrt{p_{n}}-\sqrt{p}||Z_{n}|]\\
 & \leq C_{\varphi}\expt\Big[\big(1+2^{k+1}|x|^{k}+2^{k}|Z_{n}|^{k}(\underbrace{|\sqrt{p_{n}}|^{k}+|\sqrt{p}|^{k}}_{\leq2})\big)|Z_{n}|\Big]|\sqrt{p_{n}}-\sqrt{p}|
\end{align*}
For the expectation part in the last line, by letting $n$ be large
enough so that 
\[
\big(\expt[|Z_{n}|^{2}]\big)^{\frac{1}{2}}\leq\big(\expt[|X|^{2}]\big)^{\frac{1}{2}}+1=\big(E[|\CN(0,\sdr^{2})|^{2}\big)^{\frac{1}{2}}+1=\sdr+1\eqqcolon K_1,
\]
and
\[
\big(\expt[|Z_{n}|^{2k}]\big)^{\frac{1}{2}}\leq\big(\expt[|X|^{2k}]\big)^{\frac{1}{2}}+1=\big(E[|\CN(0,\sdr^{2})|^{2k}\big)^{\frac{1}{2}}+1=\sdr^{k}((2k-1)!!)^{\frac{1}{2}}+1\eqqcolon K_2,
\]
we have
\begin{align*}
\expt\Big[\big(1+2^{k+1}|x|^{k}+2^{k+1}|Z_{n}|^{k}\big)|Z_{n}|\Big] & \leq\big(\expt[(1+2^{k+1}|x|^{k}+2^{k+1}|Z_{n}|^{k})^{2}]\big)^{\frac{1}{2}}\big(\expt[|Z_{n}|^{2}]\big)^{\frac{1}{2}}\\
 & \leq\Big(1+2^{k+1}|x|^{k}+2^{k+1}\big(\expt[|Z_{n}|^{2k}]\big)^{\frac{1}{2}}\Big)\big(\expt[|Z_{n}|^{2}]\big)^{\frac{1}{2}}\\
 & \leq C_2(1+|x|^k)
\end{align*}
 where $C_{2}=K_1\cdot\max\{1+2^{k+1}K_2,2^{k+1}\}$. Again, for the
$p_{n}$ part, we have $|p_n-p|=O(\frac{1}{n\sqrt{p}})$.
Therefore, when $k\geq 1$, for a fixed $x$, 
\[
(2)_{1}\text{or}(2)_{2}\leq C_{3}(1+|x|^{k})O(\frac{1}{n\sqrt{p}})
\]
where $C_{3}=C_{\varphi}C_{2}$. In a word, for $k\in\mathbb{N}$, 
\[
(2)_{1}\text{or}(2)_{2}\leq C_{4}(1+|x|^{k})O(\frac{1}{n\sqrt{p}})
\]
where $C_4=\max\{C_1,C_3\}$

In $(2)_3$, we can directly apply the convergence rate of nonlinear CLT (Theorem \ref{$G$-CLT-Rate}):
\begin{align*}
	(2)_3 & =|\expt[\varphi(x+\sqrt{p}\frac{1}{\sqrt{\floor{np}}}\sum_{j=0}^{\lfloor np\rfloor}X_{j})]-\expt[\varphi(x+\sqrt{p}\frac{1}{\sqrt{\floor{np}}}\sum_{j=0}^{\lfloor np\rfloor}W_{j})]| \\
	& = |\expt[\tilde{\varphi}(\frac{1}{\sqrt{\floor{np}}}\sum_{j=0}^{\floor{np}}X_{j})]- \expt[\tilde{\varphi}(\frac{1}{\sqrt{\floor{np}}}\sum_{j=0}^{\floor{np}}W_{j})] = C_{\tilde{\varphi}} O(\frac{1}{(np)^{\alpha/2}})
\end{align*}
where $\tilde{\varphi}(a)\coloneqq\varphi(x+\sqrt{p}a)$, satisfying
\begin{align*}
|\tilde{\varphi}(a)-\tilde{\varphi}(b)| & \leq C_{\varphi}(1+|x+\sqrt{p}a|^{k}+|x+\sqrt{p}b|^{k})\sqrt{p}|a-b|\\
 & \leq C_{\varphi}(1+2^{k}(|x|^{k}+|\sqrt{p}a|^{k})+2^{k}(|x|^{k}+|\sqrt{p}b|^{k})|a-b|\\
 & \leq C_{\varphi}(1+2^{k+1}|x|^{k}+2^{k}(|a|^{k}+|b|^{k}))|a-b|\\
 & \leq C_{\tilde{\varphi}}(1+|a|^{k}+|b|^{k})|a-b|
\end{align*}
in which $C_{\tilde{\varphi}}=C_{\varphi}\max\{1+2^{k+1}|x|^{k},2^{k}\}\leq C_{\varphi}(1+2^{k}+2^{k+1}|x|^{k})$. Hence, 
\[
(2)_3 \leq C_5 (1+|x|^k) O(\frac{1}{(np)^{\alpha/2}})
\]
where $C_5=C_\varphi 2^{k+1}$. 
To summarize, with $\alpha\in(0,1)$, for a given $p$ and fixed $x$, the error is
\begin{align*}
	|u(1-p,x)-\varphi_{\floor{np},n}(x)| &\leq (1)+(2) \\
	&\leq c_x O(\frac{1}{n}) + 2C_{4}(1+|x|^{k})O(\frac{1}{n\sqrt{p}}) + C_5 (1+|x|^k) O(\frac{1}{(np)^{\alpha/2}})\\
	&=M (1+|x|^k) O(\frac{1}{(np)^{\alpha/2}}).
\end{align*}
Since $p\in(0,1)$, we only need to consider what happens when $p$ approaches to $0$, in order to get a bound similar with a uniform bound, by letting $q_n\coloneqq\frac{1}{\sqrt{n}}$, we have
\[
\sup_{p>q_n}|u(1-p,x)-\varphi_{\floor{np},n}(x)|\leq M (1+|x|^k) O(\frac{1}{(\sqrt{n})^{\alpha/2}}).
\]
\end{proof}

\begin{re} This \emph{iterative algorithm} actually allows us to not only approximate the $\expt[\varphi(X)](\approx \varphi_{n,n}(0))$ but also solve the $G$-heat equation ($u(1-p,x) \approx \varphi_{\lfloor np\rfloor,n}(x)$) for any $\varphi\in C_{l.Lip}(\R)$ without involving any PDE techniques. (In the literature, for instance, the explicit solutions in \cite{hu2009explicit} come from special PDE techniques suitable for functions like $\varphi(x)=x^{2m+1},m\in\N_+$ which can be extended to functions satisfying $\varphi(\lambda x)=\lambda^\alpha \varphi(x)$ with $\alpha>0$.)\end{re}

\subsection{Extension to the $d$-dimensional Situation}

The definition of semi-$G$-normal distribution can be naturally extended
to multi-dimensional situation. Intuitively speaking, the multivariate semi-$G$-normal distribution can be treated as an analogue of the linear multivariate normal distribution: $N(\myvec 0,\mymat{V}) = \mymat{V}^{\frac{1}{2}}N(\myvec 0,\idtymat_{d})$, where $\idtymat_{d}$ is a $d\times d$ identity matrix.  

Actually, the multivariate semi-$G$-normal distribution preserve many similar properties to classical one. 
For instance, under a special sequential independence setting, we can construct the multivariate semi-$G$-normal distribution from several univariate semi-$G$-normal distributed random variables. 
Furthermore, we can use the multivariate semi-$G$-normal distribution (constructed from univariate ones) to approach the multivariate $G$-normal distribution from the nonlinear CLT. More exploration can be found in the master's thesis \cite{Li2018Stat}. 

\begin{defn}[the Semi-$G$-normal distribution in $d$ dimension]

Let a bounded, closed and convex subset $\dvarset\subset\myset{S}_d$
be the uncertainty set of covariance matrices, i.e. 
\[
\dvarset\coloneqq\{\mymat{V}=\left(\rho_{ij}\sigma_{i}\sigma_{j}\right)_{d\times d}:\sigma_{i}^{2}\in[\sdl_{i}^{2},\sdr_{i}^{2}],\rho_{ij}=\rho_{ji}=\begin{cases}
1 & i=j\\
\in[\underline{\rho}_{ij},\overline{\rho}_{ij}] & i\neq j
\end{cases},i,j=1,2,\dotsc,d \}
\]
and $\dvarset^{\frac{1}{2}}:=\{\mymat{V}^{\frac{1}{2}}:\,\mymat{V}\in \dvarset\}$ where $\dsd$ is the symmetric square root of $\dvar$. We say a $d$-dimensional
random vector $W:\Omega\rightarrow\R^{d}$ will follow the
\emph{Semi-$G$-normal distribution}, denoted as, $W\sim\hat{N}(\myvec 0,\dvarset)$,
if there exist a $d$-dimensional $G$-normal distributed random vector
\[
Y\sim N(\myvec 0,\idtymat_d):\Omega\rightarrow\R^{d},
\]
and a
$d\times d$-dimensional maximal distributed random matrix 
\[
Z\sim M(\dsdset):\Omega\rightarrow\R^{d\times d},
\]
as well as $Y$ is independent from $Z$, such that 
\[
W=Z\cdot Y
\]
where ``$\cdot$'' is the matrix multiplication (which can be omitted) and the direction of independence here
cannot be reversed.

\end{defn}
 
\begin{re}
$Y$ can be regarded as the classical multivariate normal distribution
with identity covariance matrix. 
\end{re}

\begin{cor}[the Integral Representation of the Semi-$G$-normal distribution
in $d$ dimension] \label{integral-rep-d}

Consider a $d$-dimensional
random vector $W\sim\hat{N}(\myvec 0,\dvarset)$, where $\dvarset$ is
the uncertainty set of covariance matrices. Then for any $\varphi\in C_{l.Lip}(\R^d)$,
we have 
\begin{align*}
\expt[\varphi(W)] & =\max_{\mymat{V}\in \dvarset}E[\varphi(N(\myvec 0,\mymat{V}))] =\max_{\mymat{V}^{\frac{1}{2}}\in \dsdset}E[\varphi(\mymat{V}^{\frac{1}{2}}N(\myvec 0,\mathbb{I}_{d}))]\\
 & = \max_{\mymat{V}^{\frac{1}{2}}\in \dsdset}\int_{\R^d}
\frac{1}{(2\pi)^{\frac d2}}\exp(-\frac 12 \myvec{y}'\myvec{y})\varphi(\mymat{V}^{\frac{1}{2}}\myvec y)\diff \myvec{y}.
\end{align*}
\end{cor}

Similarly, we can obtain the iterative approximation in multi-dimensional
case.

\begin{thm}[The Iterative Approximation of the $G$-normal Distribution in $d$ dimension] \label{iterate-G} Consider a $G$-normal distributed random variable
$X\sim N(\myvec 0, \dvarset)$. For any $\varphi\in C_{l.Lip}(\R^d)$ and integer $n \geq 1$, consider the series of iteration functions 
$\{\varphi_{i,n}\}_{i=1}^n$ with initial function $\varphi_{0,n}(\myvec x)\coloneqq\varphi(\myvec x)$ and iterative relation: 
\[
\varphi_{i+1,n}(\myvec x)\coloneqq \max_{\dvar \in \dvarset}E[\varphi_{i,n}(N(\myvec x,\dvar/n))],i=0,1,\dotsc,n-1.
\]
The final iteration function for a given $n$ is $\varphi_{n,n}$. As $n\to\infty$, we have $\varphi_{n,n}(\myvec 0)\to\expt[\varphi(X)]$. 

\end{thm}

\section{Implementation\label{sec:Implementation}}

In this section, we will show the implementation
of the iterative algorithm which provides a feasible way to approach $\expt[\varphi(\GN(0,\varI))]$ by using $E[\varphi(N(x, v^2))]$.

\subsection{The $1$-dimensional Situation}

Consider a $G$-normal
distributed random variable $X\sim N(0,\varI)$ with $[\sdl,\sdr]=[0.5,1]$. For any $\varphi\in C_{l.Lip}(\R)$,
in order to compute $\expt[\varphi(X)]$, setting a fixed large $n$ as the total number of iterations, we implement the following procedure: 
\begin{enumerate}
\item Start from
\[
\varphi_{0,n}(x)\coloneqq\varphi(x);
\]

\item Since we are iterating the functions on the infinite domain $\R$, in practice, we need to set up a finite grid to do interpolation at each step.
Choose a large constant $K$ to decide the range of the numerical
domain of $x$ then set up the spatial grid: 
\[
-K=x_{0}<x_{1}<x_{2}<\dotsc<x_{L}=K;
\]

\item At the iteration step $i(=1,2,\dotsc,n-1)$, for each $x=x_{j}$, $j=0,1,\dotsc,L$, evaluate
\[
\varphi_{i+1,n}(x_j)\coloneqq\max_{v\in[\sdl,\sdr]}E[\varphi_{i,n}(N(x_j,\frac{v^{2}}{n}))]
\]
where the linear expectation can be computed from integration or MC (Monte Carlo) method (by generating a linearly i.i.d.~standard
normal sample: $Z_{1},Z_{2},\dotsc,Z_{M}\sim N(0,1)$): 
\begin{align*}
E[\varphi_{i,n}(N(x,\frac{v^{2}}{n}))] & =\int_{-\infty}^{+\infty}\frac{1}{\sqrt{2\pi}}\frac{\sqrt{n}}{v}\exp(-(\frac{\sqrt{n}m}{v})^{2}/2)\varphi_{i,n}(x+m)\diff m\\
 & \approx\frac{1}{M}\sum_{k=1}^{M}\varphi_{i,n}(x+\frac{v}{\sqrt{n}}Z_{k}).
\end{align*}
Then take maximum of $E[\varphi_{i,n}(N(x,\frac{v^{2}}{n}))]$ over
$v\in[\sdl,\sdr]$ to get $\varphi_{i+1,n}(x)$ by doing appropriate optimization: here we use a quasi-Newton method called the ``L-BFGS-B''
by \cite{byrd1995limited} which ``allows box constraints, that is each variable
can be given a lower and/or upper bound''. Then use $\varphi_{i,n}(x_{j})$,
$j=0,1,\dotsc,L$, to fit the function $\varphi_{i,n}$ by choosing proper splines to do both interpolation and extrapolation based on the type of $\varphi$. For instance, for $\varphi(x)=x^3$, we 
can use the ``fmm'' method by \cite{forsythe1977computer}: ``an exact cubic spline is fitted through the four points
at each end of the data, and this is used to determine the end conditions. ''
\end{enumerate}
Finally, we have 
\[
\expt[\varphi(X)] \approx\max_{v\in[\sdl,\sdr]}E[\varphi_{n-1,n}(N(0,v^{2}))] =\varphi_{n,n}(0). 
\]

\begin{com} In the step 3, we notice that it is necessary to use
these values $\varphi_{i,n}(x+\frac{v}{\sqrt{n}}Z_{k})$, $k=1,2,\dotsc,M$,
with $M$ points in the neighbourhood of one points $x$, to estimate
the $\varphi_{i,n}(x)$. If we use a larger grid (with $M^{i}$
points) for $\varphi_{i,n}$ to preserve the precision of $\varphi_{i,n}$
on a smaller grid (with $M^{i-1}$ points), we will be stuck into
the so-called ``nested situation'' unless we are dealing with functions
with bounded domain. As we increase $n$, namely, the number of iterations,
even if we only want to compute one point of the last iterative function
with certain precision, in the previous iterations, we will still
need to prepare a series of grids which is enlarged exponentially
with respect to the iteration step. Therefore, one crucial step here
is the interpolation (and extrapolation) to avoid the nested dilemma,
which can help us get a function with continuous domain from a fixed
discrete grid. The splines should be chosen based on the type of $\varphi$ (polynomial, periodic and so on). Meanwhile, we need to appropriately choose the constant
$K$ to determine the numerical domain, to make the spline model able
to \emph{capture} the pattern of the function so as to achieve the
best extrapolation and interpolation performance. \end{com}

\begin{com}
	When evaluating the linear expectation $E[\varphi_{i}(N(x,\frac{v^{2}}{n}))]$ (omitting $n$ for a while), 
	to preserve the precision regardless of $x$, we can apply the MC with control variables. 
One common problem of Monte Carlo method is that given a normal sample,
the error of the MC estimation will be enlarged when $x$ goes farther away from 
zero. Specifically, let $h\coloneqq\frac 1n$ and $Z\sim N(0,1)$, then 
\begin{eqnarray*}
 & E[\varphi_{i}(N(x,\frac{v^{2}}{n}))]\approx \frac{1}{M}\sum_{m=1}^{M}\varphi_{i}(x+\sqrt{h}v Z_{i})=\\
 & \frac{1}{M}\sum_{m=1}^{M}\left(\varphi_{i}(x)+\frac{\varphi_{i}^{(1)}(x)}{1}\sqrt{h}v Z_{i}+\frac{\varphi_{i}^{(2)}(x)}{2}hv^{2}Z_{i}^{2}+\frac{\varphi_{i}^{(3)}(\xi_{x,Z_{i}})}{6}h^{\frac{3}{2}}v^{3}Z_{i}^{3}\right)
\end{eqnarray*}
Hence, the error can be expressed as, 
\begin{eqnarray*}
 & \frac{1}{M}\sum_{m=1}^{M}\varphi_{i}(x+\sqrt{h}v Z_{i})-E[\varphi_{i}(x+\sqrt{h}v Z)]\approx\\
 & \frac{\varphi_{i}^{(1)}(x)}{1}\sqrt{h}v(\frac{1}{M}\sum_{m=1}^{M}Z_{i}-0)+\frac{\varphi_{i}^{(2)}(x)}{2}hv^{2}(\frac{1}{M}\sum_{m=1}^{M}Z_{i}^{2}-1)
\end{eqnarray*}
Suppose $\epsilon_{1}\coloneqq\frac{1}{M}\sum_{m=1}^{M}Z_{i}-0$ and
$\epsilon_{2}\coloneqq\frac{1}{M}\sum_{m=1}^{M}Z_{i}^{2}-1$, for
a given sample $\{Z_{i}\}_{i=1}^{M}$, $\epsilon_{1}$ and $\epsilon_{2}$
are fixed number, and even if we regenerate the $\{Z_{i}\}_{i=1}^{M}$
each time, the random variable $\epsilon_{1}$ and $\epsilon_{2}$ should
also not have so much variation because we know 
\[
\epsilon_{1}\sim N(0,\frac{1}{M}),\;\epsilon_{2}\sim\frac{1}{M}\chi_{(M)}^{2}
\]
then $\text{Var}[\epsilon_{1}]=\frac{1}{M}$ and $\text{Var}[\epsilon_{2}]=\frac{2}{M}$.
However, for the case $\varphi_{1}(x)=x^{3}$, when $x$ goes farther
away from zero, the values of $|\varphi_{1}^{(1)}(x)|=3|x|^{2}$ and
$|\varphi_{1}^{(2)}(x)|=6|x|$ will become larger. Therefore, without
loss of generality, for a fixed normal sample, the error of the MC will increase as $x$ goes
away from zero. This problem can not be overcome by simply increase
the $M$ (since the $|\varphi_{1}^{(1)}(x)|$ and $|\varphi_{1}^{(2)}(x)|$
will always expand the small error anyway and these
enlarged errors will be cumulated as time goes further backward). Fortunately,
we may try a type of variance reduction method for Monte Carlo method,
called the \textbf{control variables}. We can use this method
to preserve the consistent precision of MC method outside the neighbourhood
of $x$. The idea is that, after getting the $\varphi_{1}^{(i)}(x),i=1,2$, approximate the $E[\varphi(x+\sqrt{h}v Z)]$ by
\begin{eqnarray*}
 & E[\varphi(x+\sqrt{h}v Z)-\varphi(x)-\varphi(x)-\varphi^{(1)}(x)\sqrt{h}v Z-\frac{\varphi^{(2)}(x)}{2}hv^{2}Z^{2}]+\varphi(x)+\frac{\varphi^{(2)}(x)}{2}hv^{2}\\
 & \approx\frac{1}{M}\sum_{i=1}^{M}[\varphi(x+\sqrt{h}v Z_{i})-\varphi(x)-\varphi^{(1)}(x)\sqrt{h}v Z_{i}-\frac{\varphi^{(2)}(x)}{2}hv^{2}Z_{i}]+\varphi(x)+\frac{\varphi^{(2)}(x)}{2}hv^{2}\\
 &(=\frac{1}{M}\sum_{i=1}^{M}[\frac{\varphi_{i}^{(3)}(\xi_{x,Z_{i}})}{6}h^{\frac{3}{2}}v^{3}Z_{i}^{3}]+\varphi(x)+\frac{\varphi^{(2)}(x)}{2}hv^{2}).
\end{eqnarray*} 
In this way, because of the boundedness of $\varphi^{(3)}(x)(=6)$ for $\varphi(x)=x^3$, the error will not be enlarged by $\varphi_{i}^{(3)}(\xi_{x,Z_{i}})$ when $x$ moves farther away zero. 
\end{com}

 

\begin{re} Let $X\sim \GN(0,\varI)$. In the context of the $1$-dimensional $G$-heat equation:
\[
	u_t+G(\partial_x^2 u)=0,\,u|_{t=1}=\varphi
\]
where $G(a)\coloneqq \frac{1}{2}\expt[aX^{2}]=\frac 12(\sdr^2 a^+-\sdl^2 a^-)$. We know there is a natural connection between the $G$-heat
equation and $G$-normal distribution:
\[
u(t,x)=\expt[\varphi(x+\sqrt{1-t}X)].
\]
Then we can make a collection of the iterative functions with properties: 
\begin{enumerate}
\item $\varphi_{n,n}(0)\approx\expt[\varphi(X)]=u(0,0)$; 
\item $\varphi_{n,n}(x)\approx u(0,x)$; 
\item In general, $\varphi_{k,n}(x)\approx u(1-\frac{k}{n},x)$, for
$k=0,1,\dotsc,n$. 
\end{enumerate}
This will actually give us a surface of the solution $u(t,x)$ with
$t=1-\frac{k}{n},k=0,1,\dotsc,n$. \end{re}

In the following $2\times2$-layout figures, for the first row ($n=50,K=5$), Panel-$(1,1)$ is the approximated solution paths of the $G$-heat equation with $\varphi(x)=x^{3}=\varphi_{0,n}(x)$ (the black solid curve) and the curves whose right branches are moving up are the $\varphi_{i,n}$'s as the iteration proceeds. 
 The Panel-$(1,2)$ shows the approximated paths with $\varphi(x)=(1-|x|)\ind_{(-1,1)}(x)$ (which belongs to a class of functions that, previously, are hard to deal with if not applying special PDE methods). 
 The Panel-$(2,1)$ is the pathwise comparison plot ($n=100,K=50,\varphi(x)=x^3$) with the explicit solution provided in \citet{hu2009explicit} where we can only notice the error (with absolute value $\leq 0.004$) around $x=0$ and we can see the stable accuracy beyond the spatial grid $[-K,K]$ (the black dashed lines).The Panel-$(2,2)$ is the approximation of $\expt[X^3]$ as $n$ increases ($K=5$) 
(where the horizontal line labels the true value; here we use integration to compute $E$, if using MC, we will assign a variance for each point).  
\begin{center}
\includegraphics[width=\linewidth]{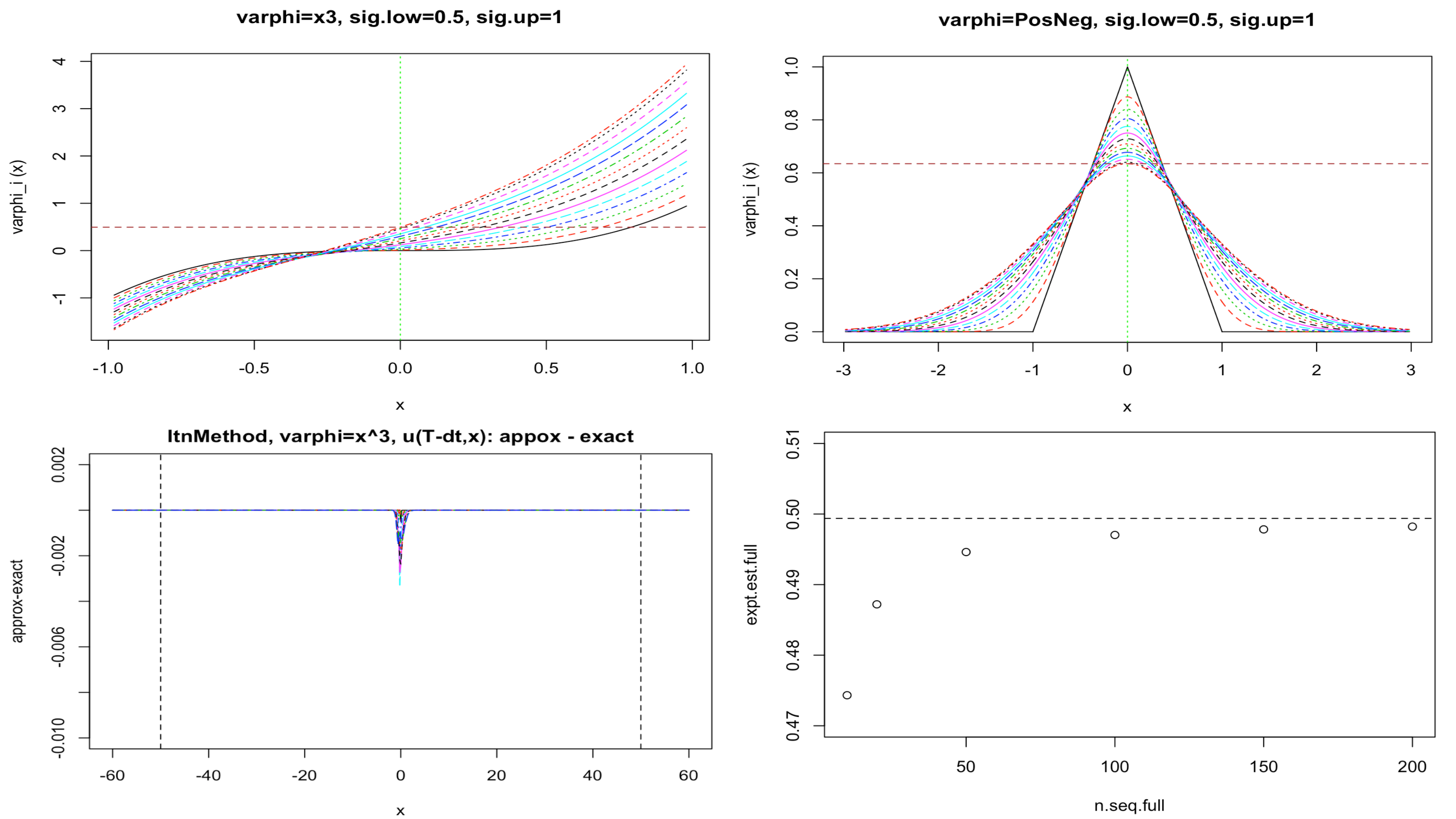} 
\par\end{center}

For curiosity, we can play with this algorithm ($n=50,K=5$) by changing the terminal function to other $\varphi\in \fspace$:  Panel-$(1,1)$ ($\varphi(x)=\sin x$), Panel-$(2,1)$ ($\varphi(x)=\cos x$), Panel-$(1,2)$ ($\varphi(x)=\sin x + \cos x$), and Panel-$(1,1)$ ($\varphi(x)=1/(1+\exp(-x^2))$).

\begin{center}
\includegraphics[width=\linewidth]{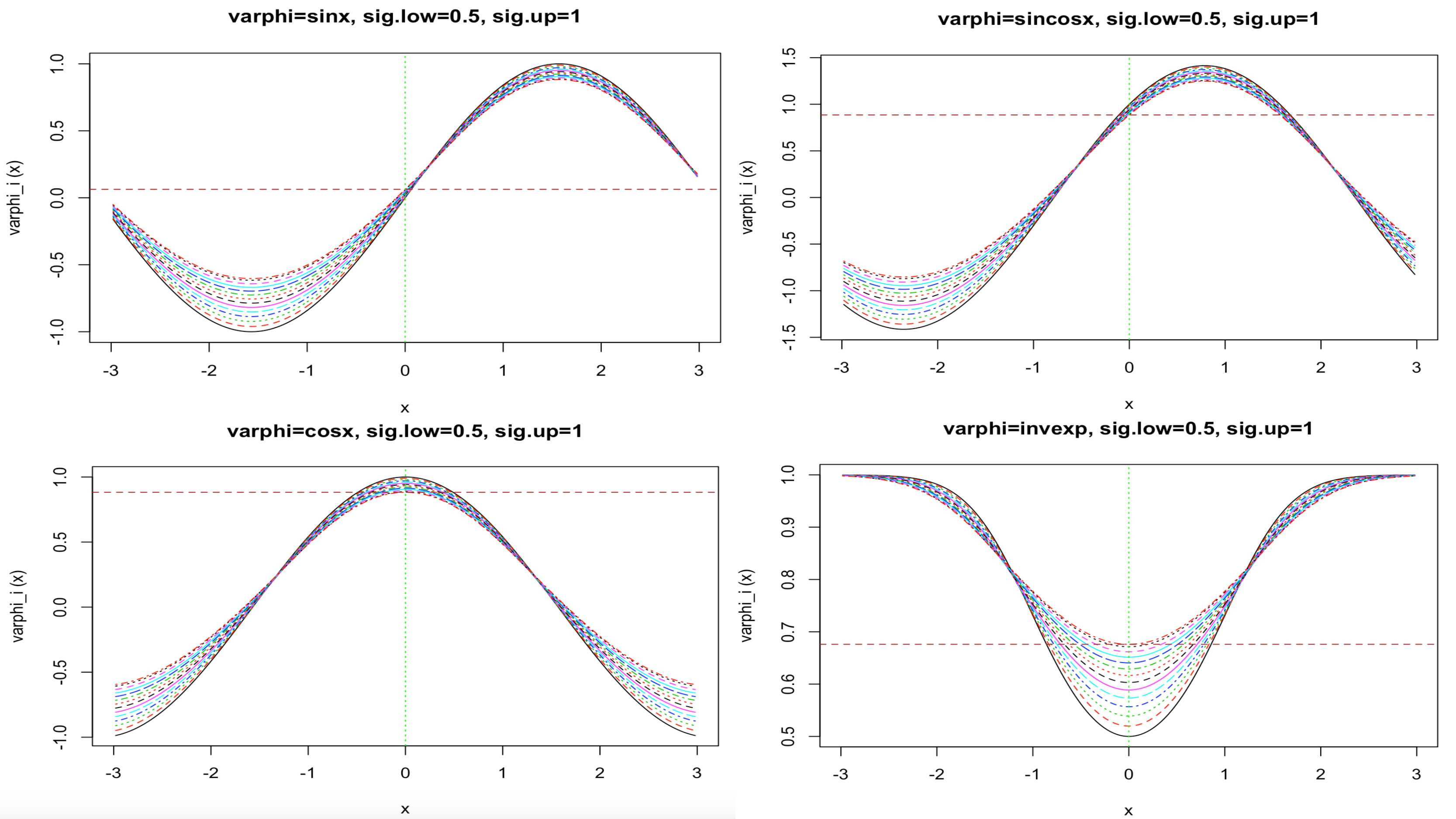} 
\par\end{center}

(I am still working on the implementation when involving the viscosity solution like the case with $\sdl=0$.)









\subsection{The $d$-dimensional Situation}

For the multi-dimensional $G$-heat equation (with covariance uncertainty), the MC setting may relieve the iterative algorithm from the curse of dimensionality. 

Let $\myvec{x}\coloneqq(x_{1},x_{2},\dotsc,x_{d})\in\R^{d}$.
Consider a $d$-dimensional
$G$-normal distributed random vector $X\sim N(\myvec{0},\dvarset)$, where
$\dvarset$ is the uncertainty set of covariance matrices. For any $\varphi\in C_{l.Lip}(\R^{d})$,
in order to compute $\expt[\varphi(X)]$, similarly, set a large $n$
as the total number of iterations, then we need to do the procedure: 
\begin{enumerate}
\item Start from 
\[
\varphi_{0,n}(\myvec{x})\coloneqq\varphi(\myvec{x});
\]

\item We need to set up a grid in the domain $\R^{d}$, to avoid
the curse of dimensionality in the sense that the rectangular grid
points become sparser in higher dimension and much more with exponential
rate ($L$ points in each dimension mean $L^{d}$ points in total). Therefore,
here we use the Monte Carlo grid points sampling from a $d$-dimensional
multivariate normal distribution: 
\[
\{\myvec{x}_{j}\}_{j=0}^{L}\sim N(\myvec{0},\sdr^{2}\idtymat_d).
\]

\item For $i=1,2,\dotsc,n$, and for each $\myvec{x}=\myvec{x}_{j}$, $j=0,1,\dotsc,L$,
let 
\[
\varphi_{i+1,n}(\myvec{x})\coloneqq\max_{\dvar\in \dvarset}E[\varphi_{i,n}(N(\myvec{x},\frac{\dvar}{n}))],
\]
where, again, to deal with the curse of dimensionality, the expectation
can be computed by Monte Carlo method which maintains its convergence
rate regardless of the dimension by generating a linearly i.i.d.~sample
from standard multivariate normal $N(0,\mathbb{I}_{d})$: $Z_{1},Z_{2},\dotsc,Z_{M}$ (and we can also apply the control variable method to protect the precision from the variation of $\myvec x$):
\[
E[\varphi_{i,n}(N(\myvec{x},\frac{\dvar}{n}))]\approx\frac{1}{M}\sum_{k=1}^{M}\varphi_{i,n}(\myvec{x}+\frac{\dvar^{\frac{1}{2}}}{\sqrt{n}}Z_{k}).
\]
Take maximum of $E[\varphi_{i-1,n}(N(\myvec{x},\frac{\dvar}{n}))]$
over $\dvar \in \dvarset$ by appropriate optimization (here
we still use the ``L-BFGS-B'' method). Then use $\varphi_{i+1,n}(x_{j})$,
$j=0,1,\dotsc,L$, to fit the function $\varphi_{i+1,n}$ by applying a proper spline model to do interpolation and extrapolation (here we work on the setting of \emph{Generalized
Additive Model} after testing and design the structure of splines based on the properties of $\varphi$). 
\end{enumerate}
Finally, we have 
\begin{align*}
\expt[\varphi(X)] & \approx\max_{\dvar^{\frac{1}{2}}\in \dsdset}E[\varphi_{n-1,n}(N(\myvec{0},\frac{\dvar}{n}))]\\
 & =\varphi_{n,n}(\myvec{0}).
\end{align*}

\begin{re} Consider a $2$-dimensional $G$-normal distributed random variable
$X\sim N(0,\dvarset)$ where 
\[
\dvarset\coloneqq\left\{ \dvar=\left(\begin{array}{cc}
\sigma_{1}^{2} & \rho\sigma_{1}\sigma_{2}\\
\rho\sigma_{1}\sigma_{2} & \sigma_{2}^{2}
\end{array}\right):\sigma_{i}\in[0.5,1],i=1,2;\rho\in[-0.5,0.5]\right\} .
\]
In the context of $2$-dimensional $G$-heat equation: 
	\[
	u_t+G(D_x^2 u)=0,\,u|_{t=1}=\varphi
	\]
	where $G(\mymat{A})\coloneqq\frac 12\expt[\langle \mymat{A}X,X \rangle]:\myfield{S}_d \to \R$,
we also have the natural connection between the $G$-heat equation
and $G$-normal distribution, 
\[
u(t,x)=\expt[\varphi(\myvec{x}+\sqrt{1-t}X)],\;(t,\myvec{x})\in[0,\infty]\times\R^{2}.
\]
Then these results have been verified: 
\begin{enumerate}
\item we have $\varphi_{n,n}(0)\approx\expt[\varphi(X)]=u(0,0)$ which
repeats the above remark; 
\item by replacing $\varphi(X)$ with a shifted version $\varphi(\myvec{x}+X)$,
we directly get $\varphi_{n,n}(\myvec{x})\approx u(0,\myvec{x})$; 
\item In general, we have $\varphi_{k,n}(\myvec{x})\approx u(1-\frac{k}{n},\myvec{x})$,
for $k=0,1,\dotsc,n$. In other words, the function in each iteration
step has its connection with the solution of $G$-heat equation. 
\end{enumerate}
\end{re}

This is the solution surface with $\varphi(\myvec{x})=x_{1}^{3}+x_{2}^{3}$,
the number of iteration steps $n=10$, and 
\[
\dvarset\coloneqq\left\{ \dvar=\left(\begin{array}{cc}
\sigma_{1}^{2} & \rho\sigma_{1}\sigma_{2}\\
\rho\sigma_{1}\sigma_{2} & \sigma_{2}^{2}
\end{array}\right):\sigma_{i}\in[0.5,1],i=1,2;\rho\in[-0.5,0.5]\right\} .
\]

\begin{center}
\includegraphics[width=1\linewidth]{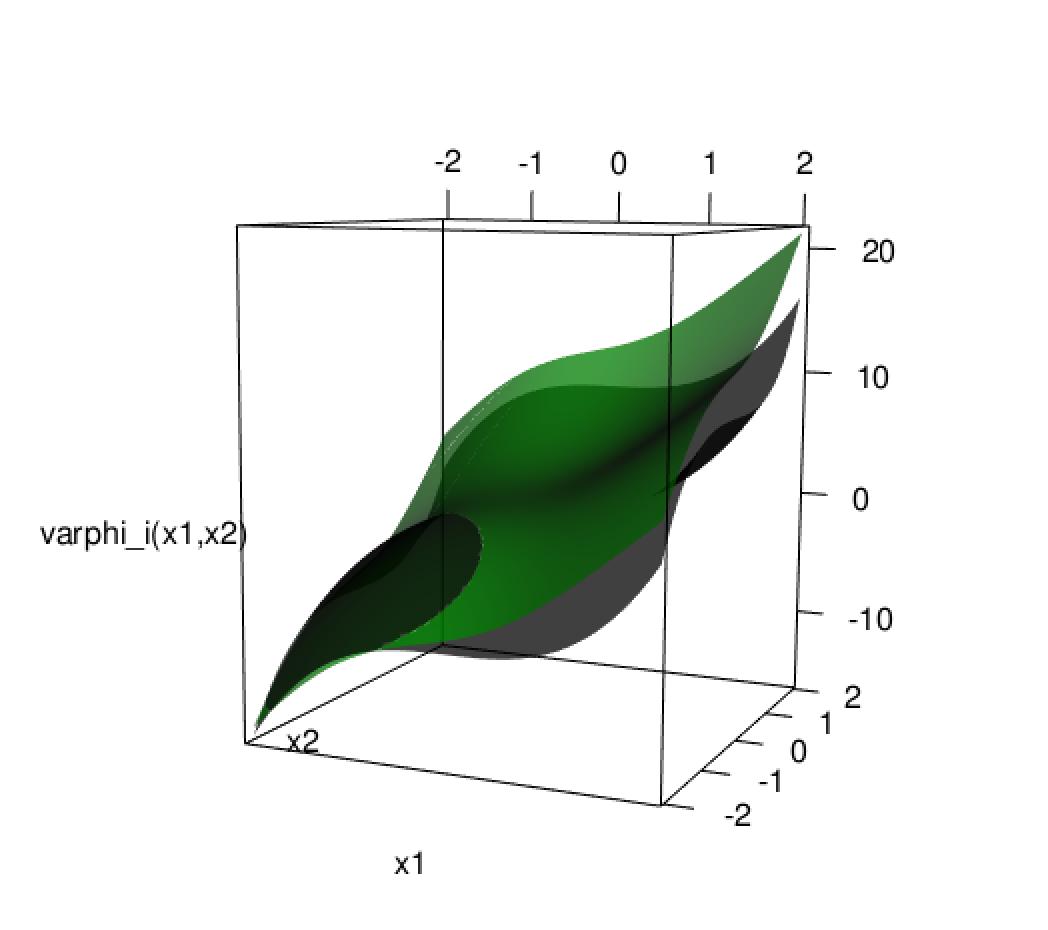} 
\par\end{center}

In the above figure, the black surface is $u(1,\myvec x)=\varphi(\myvec x)$ and the green one
is $u(0,\myvec x)$. 

(I am still working on more details for this and also the Generalized Additive Model with control variable method.)

\section{Assessment of the Iterative Method}

The strengths of this iterative method are well-worth mentioning:
\begin{enumerate}
\item It will work for any $\varphi$ applicable in the nonlinear central
limit theorem, including some irregular $\varphi$ which makes $\expt[\varphi(X)]$
difficult to compute even by using classical numerical PDE methods.
This property has been verified in the proof of Theorem \ref{iterate-G}
and has been checked numerically for one dimension. 
\item It will numerically solve the corresponding $G$-heat equation by not
just giving one point or one path but actually directly providing
the whole surface of $u(t,x)$, because the function in each iteration
step has its connection with the $u(t,x)$ at one grid point $t$.
This is given by Theorem \ref{iterate-G}; 
\item It will give us a great visualization and intuition about how the
solution surface of $G$-heat equation is evolved from the terminal function
$\varphi(x)$ by looking at the procedure of iteration, and how the
sublinear expectation of $G$-normal distribution is approached by iteratively
maximizing the linear expectation of classical normal distribution.
The inherent bridge or ladder here is the Semi-$G$-normal distribution
which helps us climb from the stage of classical normal distribution
to reach the stage of $G$-normal distribution. It partially fills in
the long-existing thinking gap between them, since we have the inequality:
$\expt[\varphi(X)]\geq\sup_{v\in\sdInt}E[\varphi(N(0,v^{2}))]$ (for
various $\varphi$, especially when $\varphi$ is neither convex nor
concave, the strict greater relation is much more frequent than equality),
which, in some sense, strictly separated the $G$-normal distribution
with the classical one and made us feel a little risky to connect
them for a long time; 
\item It can be naturally extended to higher dimension both in theory and
algorithm, since we already have the established multi-dimensional
distributions in $G$-framework, as well as the algorithm for computation
of linear expectation of classical multivariate normal distribution.
Then it can solve the corresponding multi-dimensional $G$-heat equation
attached with covariance uncertainty; 
\end{enumerate}

However, for numerical practice, we still need to reflect on how to properly choose the splines to achieve better interpolation and extrapolation performance. 

\section{Concluding Remarks and Further Development}
The semi-$G$-normal distribution proves its own value by giving us this
iterative approximation.
This is a typical distribution connecting the maximal distribution
from nonlinear framework and the classical normal distribution from
linear framework.

In the master's thesis \cite{Li2018Stat}, we further explore the semi-$G$-normal distribution by considering its independence, its multivariate version (with the construction from univariate objects) and the semi-$G$-Brownian motion (with its connection to the $G$-Brownian motion. For the statistical side, we also provide a pseudo approach to simulate the semi-$G$-normal distribution and the estimation method for the variance interval of $\semiGN(0,\varI)$. The semi-$G$-normal distribution $\semiGN(0,\varI)$ behaves like the transition from the linear normal distribution $N(0,\sigma^2)$ to the sublinear $G$-normal distribution $\GN(0,\varI)$.

The set of measures of the semi-$G$-normal distribution, consisting of a class of linear measures of $N(0,\sigma^2)$ with $\sigma\in\sdInt$, is smaller than that of the $G$-normal distribution. Hence, on the one hand, it has less unusual properties than the $G$-normal distribution and more similar properties to the linear normal distributions (e.g. its sublinear ``skewness'' is equal to zero); on the other hand, it has strong connections with $G$-normal distribution (e.g. the sublinear expectation with convexity and the nonlinear CLT). As a crucial concept in the thesis \cite{Li2018Stat}, the semi-$G$-normal distribution is aimed for constructing a bridge between the linear and the sublinear world so that in the future, more tools on the land of linear expectation framework can be transformed to the island of the sublinear expectation framework to help the larger community in both the academy and industry to understand the intuition, learn the theory or algorithms and apply them to the real world problems with uncertainty. 










%

\phantomsection 
\addcontentsline{toc}{section}{\refname}
\bibliographystyle{chicago}
\bibliography{Refs}

\end{document}